\documentclass[12pt]{amsart}

\usepackage{amsfonts}
\usepackage{graphicx}
\usepackage{tabularx}
\usepackage{array}
\usepackage[usenames,dvipsnames]{color}
\usepackage{comment}
\usepackage{amsmath}
\usepackage{amsthm}
\usepackage{amssymb}
\usepackage{fullpage}
\usepackage[dvipsnames]{xcolor}
\usepackage{tikz}
\usepackage{listings}
\usepackage{dsfont}
\usepackage{enumerate}
\usepackage{romannum}

\newtheorem{theorem}{Theorem}[section]

\newtheorem{conjecture}[theorem]{Conjecture}
\newtheorem{problem}[theorem]{Problem}

\newtheorem{lemma}[theorem]{Lemma}

\theoremstyle{definition}

\def\epsilon{\varepsilon}

\title{Cops and Robbers, Clique Covers, and Induced Cycles}
\author{Alexander Clow}
\address{Department of Mathematics, Simon Fraser University}
\email{alexander\_clow@sfu.ca}
\author{Imed Zaguia}
\address{Department of Mathematics and Computer Science, Royal Military College of Canada}
\email{zaguia@rmc.ca}
\date{\today}

\begin{document}

\begin{abstract}
    We consider the Cops and Robbers game played on finite simple graphs.
    In a graph $G$, the number of cops required to capture a robber in the Cops and Robbers game is denoted by $c(G)$.
    For all graphs $G$, 
    $c(G) \leq \alpha(G) \leq \theta(G)$
    where $\alpha(G)$ and $\theta(G)$ are the independence number and clique cover number respectively.
    In 2022 Turcotte asked if $c(G) < \alpha(G)$ for all graphs with $\alpha(G) \geq 3$.
    Recently, Char, Maniya, and Pradhan proved this is false, at least when $\alpha = 3$,
    by demonstrating the compliment of the Shrikhande graph has cop number and independence number $3$.
    We prove, using random graphs, the stronger result that for all $k\geq 1$ there exists a graph $G$ such that $c(G) = \alpha(G) = \theta(G) = k$.
    Next, we consider the structure of graphs with $c(G) = \theta(G) \geq 3$.
    We prove, using structural arguments, that any graphs $G$ which satisfies $c(G) = \theta(G) = k \geq 3$
    contain induced cycles of all lengths $3\leq t \leq k+1$.
    This implies all perfect graphs $G$ with $\alpha(G)\geq 4$
    have $c(G) < \alpha(G)$.
    Additionally,we discuss if typical triangle-free and $C_4$-free
    graphs will have
    $c(G) < \alpha(G)$.
\end{abstract}

\maketitle

\section{Introduction}
\pagenumbering{arabic}

Cops and Robbers, see \cite{AIGNER1984, nowakowski1983vertex,quilliot1983problemes}, is a two-player game played on a graph. 
To begin the game, the cop player places $k$ cops onto vertices of the graph,
then the robber player chooses a vertex to place the robber. 
Players take turns moving. 
During the cop player’s turn, each cop either moves to an adjacent vertex or passes and remains at their current vertex. 
Similarly, on the robber player’s turn, the robber either moves to an adjacent vertex or passes and remains at their current vertex.
The cop player wins if there is a cop strategy in which after finitely many moves, a cop can move onto the vertex currently occupied by the robber (called capturing).
The robber player wins if the robber can evade capture indefinitely. 
The least number of cops required for the cop player to win, regardless of the robber’s strategy, is the cop number of a graph, denoted $c(G)$ for a graph $G$.
For more background on Cops and Robbers we recommend \cite{bonato2011game}.

We define the graphs $P_t$, $C_t$, and $K_t$ as the path with $t$ vertices, the cycle with $t$ vertices, and the complete graph on $t$ vertices, respectively.
If $G$ and $H$ are graphs, then we denote the disjoint union of $G$ and $H$ by $G+H$.
For any graph $G$, we use $mG$ to denote the disjoint union of $m$ copies of $G$, and we let $\bar{G}$ denote the compliment of $G$.
We let $\alpha(G),\theta(G),\omega(G),$ and $\chi(G)$ denote the independence number, clique cover number, clique number, and chromatic number of $G$ respectively.
For more background and definitions in graph theory we refer the reader to \cite{west2001introduction}.

It is standard to consider classes of graphs defined by forbidden substructures such as minors or induced subgraphs.
A graph $G$ is $H$-free or $H$-minor free if $G$ does not
contain, respectively, any induced subgraph or minor which is isomorphic to $H$. 
Of particular importance for this paper will be the class of graphs with independence number at most $k$, 
which is equivalently the class of $(k+1)K_1$-free graphs.
Perfect graphs will also be of interest.
A graph $G$ is perfect if for all induced subgraphs $H$ of $G$, $\chi(H) = \omega(H)$.

Cops and Robbers has been studied extensively in relation to forbidden minors and forbidden induced subgraphs.
Andreae \cite{andreae1986pursuit} showed that for all graphs $H$, there exists a constant $m = m(H)$
such that if $G$ has no $H$-minor, then $c(G) \leq m$.
The constant $m = m(H)$ here was recently improved by Kenter, Meger, and Turcotte \cite{kenter2025improved},
particularly for small or sparse graphs $H$.
When forbidding an induced subgraph $H$, one cannot guarantee the cop number of $H$-free graphs is bounded.
In fact, it was shown by Joret, Kami{\'n}ski, and Theis \cite{joret2010cops} 
that $H$-free graphs have bounded cop number if and only if $H$ is a linear forest (i.e a disjoint union of paths).
This characterization of which forbidden induced subgraphs admit classes with bounded cop number was extended by Masjoody, and Stacho \cite{masjoody2020cops}
who demonstrated analogues results when forbidding multiple induced subgraphs simultaneously.

This has lead to a significant amount of effort into determining the correct upper bound on the cop number of $H$-free graphs when $H$ is a linear forest.
Joret, Kami{\'n}ski, and Theis \cite{joret2010cops} proved that for all $t\geq 3$, every $P_t$-free graph has cop number at most $t-2$.
Sivaraman \cite{sivaraman2019application} conjectured that for all $t\geq 4$, every $P_t$-free graph has cop number at most $t-3$.
This is obvious when $t=4$, and was only recently proven when $t=5$ by Chudnovsky, Norin, Seymour, and Turcotte \cite{chudnovsky2024cops}.

The focus of this paper is a related problem, proposed by Turcotte in \cite{turcotte2022cops}.
Before stating the problem we note that for all graphs $G$, $c(G) \leq \alpha(G)$ since the cops can begin the game in a largest independent, which is necessarily a dominating set.
This bound is tight for graphs with independence number at most $2$. To see this consider $C_4$.
Turcotte \cite{turcotte2022cops} asks if this upper bound can be improved to $c(G) < \alpha(G) $ when $\alpha(G)\geq 3$.

When $\alpha(G) = 3$,
Char, Maniya, and Pradhan \cite{char20254} proved that the bound $c(G) \leq \alpha(G)$ is tight.
In particular, they demonstrate a graph with $16$ vertices, see the compliment of the graph in Figure~\ref{Fig:cop=Alpha=3}, which has cop number and independence number $3$.
Notice that the graph is Figure~\ref{Fig:cop=Alpha=3} has chromatic number $4$, hence its compliment has clique cover number $4$.

\begin{figure}[!h]
    \centering
    \includegraphics[scale = 0.625]{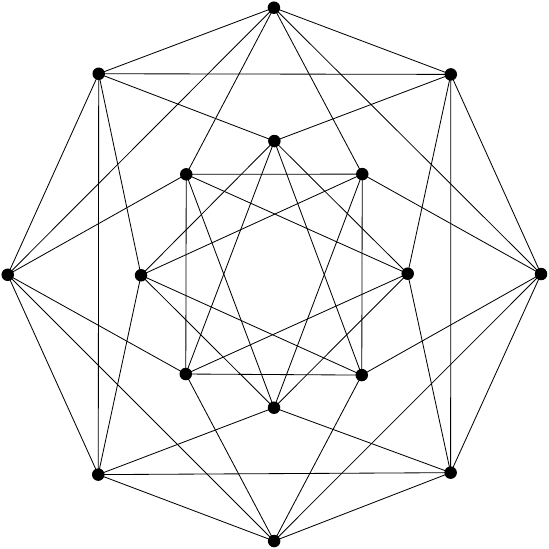}
    \caption{The Shrikhande graph is pictured. The compliment of the Shrikhande graph has cop number and independence number $3$, as well as clique cover number $4$.
    The compliment is not showns since the number of edge crossings make the figure non-instructive.}
    \label{Fig:cop=Alpha=3}
\end{figure}

Recalling that for all graphs $G$, $c(G) \leq \alpha(G) \leq \theta(G)$, we consider the existence and structure of graphs which satisfy the stronger equality $c(G) = \theta(G)$.
Of course any graph with $c(G) = \theta(G)$ must also have $c(G) = \alpha(G)$.
As a result, it is not obvious that such graphs exist when $c(G) = \theta(G)\geq 3$, 
particularly since the graph given by Char, Maniya, and Pradhan \cite{char20254} has cop number $3$ and clique cover number $4$.
Thus, we begin our work by demonstrating that for all $k\geq 1$, there exists a graph with $c(G) = \theta(G) = k$.
This resolves Turcotte's question, as it implies the bound $c(G) \leq \alpha(G)$ is tight regardless of the value of $\alpha(G)$.

\begin{theorem}\label{Thm: Cop = Clique = k for all k}
Let $k$ be a fixed but arbitrary positive integer.
There exists a graph $G$ with at most $11\log(k)^2k^4$ vertices
satisfying that $c(G) = \theta(G) = k$.
\end{theorem}

With the existence of such graphs established, we consider the cycle structure of graphs with equal cop number and clique cover number.
Specifically, we show the following.

\begin{theorem}\label{Thm: Clique Cover = Cop Structure}
    Let $k\geq 3$.
    If $G$ is a connected graph with $c(G) = \theta(G) = k$, then for all $3 \leq t \leq k+1$, $G$ contains an induced cycle of length $t$.
\end{theorem}

This theorem has some nice corollaries. 
Specifically, we are able to show cop number cannot equal independence number for certain classes, such as perfect graphs.
This shows Turcotte's question remains interesting when restricted to special classes of graphs.

Recall the weak perfect graph theorem, proven by Lov\'asz \cite{lovasz1972normal}, states the a graph $G$ is perfect if and only if $\bar{G}$ is perfect.
Thus, every perfect graph $G$ has $\alpha(G) = \theta(G)$.
Also, recall the strong perfect graph theorem, proven by Chudnovsky and Seymour \cite{chudnovsky2006strong}, which states that a graph $G$ is perfect if and only if for all $k\geq 2$, $G$ is $C_{2k+1}$-free and $\bar{C}_{2k+1}$-free.
Taking these facts together, Theorem~\ref{Thm: Clique Cover = Cop Structure} immediately implies one cannot have a perfect graph with large independence number equal to cop number.

\begin{theorem}\label{Coro: Perfect Graphs}
    If $G$ is a connected perfect graph with $\alpha(G) \geq 4$, then $c(G) < \alpha(G)$.
\end{theorem}

A similar, albeit weaker, question:
for which classes of graphs $\mathcal{G}$ does a `typical' graph $G \in \mathcal{G}$ have $c(G) < \alpha(G)$?
Typical here refers to almost every graph in $\mathcal{G}$ satisfying this property.
Recently, Reed and Yuditsky \cite{reed2025asymptotic} proved that if $H$ is a tree, or a cycle $C_\ell$ where $\ell\neq 6$, then 
almost every $H$-free graph $G$ has $\chi(G) = \omega(G)$.
If the same result is true for forests $F$ and $F$-free graphs, 
then we can combine Theorem~\ref{Thm: Clique Cover = Cop Structure} along with some easy to prove lemmas
to show the following result.
A proof is provided in Section~4.

\begin{theorem}\label{Coro: typical C_l-free Graphs}
    Assume that almost all $3K_1$-free graphs and $2K_2$-free graphs have $\chi = \omega$.
    Let $\ell$ equal $3$ or $4$.
    Almost every $C_\ell$-free graph $G$ has $c(G) < \alpha(G)$.
\end{theorem}

The rest of the paper is structured as follows.
In Section~2 we prove Theorem~\ref{Thm: Cop = Clique = k for all k} using random graphs.
Next, in Section~3 we prove Theorem~\ref{Thm: Clique Cover = Cop Structure} using a structural argument.
Finally, we prove Theorem~\ref{Coro: typical C_l-free Graphs} in Section~4.
We conclude with a discussion of future work.

\section{There Exists Graphs with $c(G) = \theta(G)$}

In this section we will prove Theorem~\ref{Thm: Cop = Clique = k for all k}.
Our proof uses random graphs.
To this end we define the following random graph, which we denote $H(k,\ell,p)$.
Let $k$ and $\ell$ be integers, and let $p \in (0,1)$ be a real number.
Let $V(H(k,\ell,p))$ be the union of $k$ disjoint sets $S_1,\dots, S_k$ all of whom have cardinality $\ell$.
We suppose that each set $S_i$ induces a clique in $H(k,\ell,p)$.
For all pairs $u \in S_i$ and $v\in S_j$ such that $i\neq j$, 
the probability that edge $(u,v)$ is in $H(k,\ell,p)$ is $p$, and this event is independent from the event of any other edge being in $H(k,\ell,p)$.

Before proceeding further, we recall a useful and standard probabilistic inequality.
Inequalities of this type are often called Chernoff bounds.
Here $\exp(z) := e^z$.

\begin{lemma}\label{Lemma: Chernoff's Bound}\emph{[\cite{janson2011random}]}
    Let and $0 < \epsilon \leq \frac{3}{2}$ be a real number and let $X$ be a random variable with binomial distribution $Bin(n,p)$ where $0<p<1$. Then, 
    \[
    \mathbb{P}\Big(|X - \mathbb{E}(X)| \geq \epsilon\mathbb{E}(X)\Big) \leq \exp\Bigg(-\frac{\mathbb{E}(X)\epsilon^2}{3}\Bigg).
    \]
\end{lemma}

Using this Chernoff bound we can calculate the clique cover number of $H(k,\ell,p)$ with high probability
for certain useful values of $k,\ell,$ and $p$.
To do this, we first consider the maximum degree of $H(k,\ell,p)$.
If the base of a logarithm is not specified then assume it is the natural logarithm.

\begin{lemma}\label{Lemma: Max-degree random}
    Let $k$ be a fixed constant.
    For all $\ell$ such that $5k(k-1)\log(\ell) \leq \frac{\ell}{k-1}$,
    \[
    \mathbb{P}\Bigg(\Delta(H(k,\ell,p)) < \frac{k\ell}{k-1} - 1 \Bigg) \geq 1 - k\ell^{1-\frac{3(k-1)k}{2}}
    \]
    when $p = \frac{2k\log(\ell)}{\ell}$.
\end{lemma}

\begin{proof}
    Let $k$ be a fixed constant, let  $5k(k-1)\log(\ell)  \leq \frac{\ell}{k-1}$, and let $p = \frac{2k\log(\ell)}{\ell}$.
    Let $H = H(k,\ell,p)$.
    Consider the event that $\Delta(H) \geq \frac{k\ell}{k-1}  - 1$.
    For each vertex $v$ in $H$ let $X_v$ be the random variable that counts the degree of $v$ in $H$.
    Observe that $X_v = \ell - 1 + Y$ where $Y$ is a binominal random variable with the same distribution as $Bin((k-1)\ell,p)$.
    Then applying the union bound,
    \begin{align*}
        \mathbb{P}\Bigg(\Delta(H) \geq \frac{k\ell}{k-1}  - 1 \Bigg) & \leq |V(H)| \mathbb{P}\Bigg(X_v \geq \frac{k\ell}{k-1}  - 1\Bigg)\\
        & = k\ell \mathbb{P}\Bigg(Y \geq \frac{\ell}{k-1}\Bigg).
    \end{align*}
    Thus, to bound the probability of $\Delta(H) \geq \frac{k\ell}{k-1}  - 1$ it is sufficient to bound the probability that $Y \geq \frac{\ell}{k-1}$.

    Recalling that $Y$ has the same distribution as $Bin((k-1)\ell,p)$,
    the expected value of $Y$ is $p(k-1)\ell = 2(k-1)k\log(\ell)$.
    By assumption $5k(k-1)\log(\ell) \leq \frac{\ell}{k-1}$, so
    applying Lemma~\ref{Lemma: Chernoff's Bound} with $\epsilon = \frac{3}{2}$ gives that,
    \begin{align*}
        \mathbb{P}\Bigg(Y\geq \frac{\ell}{k-1} \Bigg)
        & \leq  \mathbb{P}\Big(Y\geq 5k(k-1)\log(\ell) \Big) \\
        & \leq \mathbb{P}\Bigg(|Y - 2k(k-1)\log(\ell) | \geq \frac{3}{2} \Big(2(k-1)k\log(\ell) \Big) = 3(k-1)k\log(\ell) \Bigg) \\
        & \leq \exp\Bigg(-\frac{3(k-1)k\log(\ell)}{2}\Bigg) \\
        & = \ell^{-\frac{3(k-1)k}{2}}
    \end{align*}
    Plugging this into the earlier union bound gives 
    \begin{align*}
        \mathbb{P}\Bigg(\Delta(H) \geq \frac{k\ell}{k-1}  - 1 \Bigg) & \leq k\ell \Big( \ell^{-\frac{3(k-1)k}{2}}\Big) \\
        & = k\ell^{1-\frac{3(k-1)k}{2}}.
    \end{align*}
    This concludes the proof.
\end{proof}

\begin{lemma}\label{Lemma: Clique = k in H}
    Let $k$ be a fixed constant and let $\ell$ and $p$ be fixed but arbitrary.
    If the graph $H(k,\ell,p)$ has maximum degree 
    \[
    \Delta(H(k,\ell,p)) < \frac{k\ell}{k-1} - 1
    \]
    then $\theta(H(k,\ell,p)) = k$.
\end{lemma}

\begin{proof}
 Let $k$ be a fixed constant and let $\ell$ and $p$ be fixed but arbitrary.
Let $H = H(k,\ell,p)$.
To begin we note that trivially $\theta(H)\leq k$, since $S_1,\dots, S_k$ is a partition of the vertices of $H$ such that each $S_i$ induces a clique.
So it suffice to show that $\theta(H)> k-1$ if $\Delta(H) < \frac{k\ell}{k-1} - 1$.

Suppose that $H$ has $\theta(H) \leq k-1$.
Since $\theta(H) = \chi(\bar{H})$ and $\omega(H) = \alpha(\bar{H})$
we note that 
\begin{align*}
    k - 1 & \geq \theta(H) = \chi(\bar{H}) \geq \frac{|V(H)|}{\alpha(\bar{H})} = \frac{k\ell}{\omega(H)}
\end{align*}
implies that $\omega(H) \geq \frac{k\ell}{k-1}$.
Trivially, if $\omega(H) \geq \frac{k\ell}{k-1}$, then $\Delta(H) \geq \frac{k\ell}{k-1}  - 1$, since any vertex in a clique of size $\frac{k\ell}{k-1}$ must have degree at least $\frac{k\ell}{k-1}  - 1$.
Therefore, if $H$ has $\Delta(H) < \frac{k\ell}{k-1} - 1$,
then $\omega(H) < \frac{k\ell}{k-1}$, which implies $\theta(H) > k-1$.
This
concludes the proof.
\end{proof}

Next, we prove that for all $i$, every vertex $v\in S_i$ has a useful adjacency property with vertices outside of $S_i$ with high probability.
We will use this adjacency property in $H(k,\ell,p)$ to lower bound cop number.
Our approach to lower bounding cop number is similar to many other arguments in the cops and robber literature.

\begin{lemma}\label{Lemma: Always an escape in H}
    Let $k$ be a fixed constant, let $\ell\geq 2$ be fixed but arbitrary, and let $p = \frac{2k\log(\ell)}{\ell}$.
    If $v\in S_i$
    let $E_v$ be the event that there exists an integer $j\neq i$ and a set of vertices $u_1,\dots, u_{k-1} \in V(H(k,\ell,p))\setminus (S_j \cup \{v\})$
    such that 
    \[
    N(v) \cap S_j \subseteq \Big(\bigcup_{t=1}^{k-1} N(u_t) \Big).
    \]
    Then, the probability that no event $E_v$ occurs is at least 
    \[
    1 - \frac{k^{k+1}}{(k-1)!}\exp\Big(4 k^{3} \log(\ell)^2 \ell^{(\frac{1}{\ell}-1)}\Big)\ell^{-k}.
    \]
\end{lemma}

\begin{proof}
    Let $k$ be a fixed constant, let $\ell\geq 2$ be integer valued, let $p = \frac{2k\log(\ell)}{\ell}$,
    and 
    let $i$ and $v \in S_i \subseteq V(H(k,\ell,p))$ be fixed but arbitrary.
    For a fixed $j\neq i$, let $E_{v,j}$ be the event
    that there exists a set of vertices $u_1,\dots, u_{k-1} \in V(H(k,\ell,p))\setminus (S_j \cup \{v\})$
    such that 
    \[
    N(v) \cap S_j \subseteq \Big(\bigcup_{t=1}^{k-1} N(u_t) \Big).
    \]
    Note that by the symmetry of $H(k,\ell,p)$ the choice of $j$ does not effect the probability of the event $E_{v,j}$.
    Then, by the union bound $\mathbb{P}(E_v) \leq (k-1)\mathbb{P}(E_{v,j})$.
    So to bound the probability of $E_v$ it is sufficient to bound the probability of $E_{v,j}$ for a fixed $j$.

    Let $j$ be fixed.
    Let $X$ be the random variable that counts the cardinality of $N(v) \cap S_j$.
    Then $X$ has distribution $Bin(\ell,p)$.
    Hence,
    \begin{align*}
        \mathbb{P}(X = x) & = \binom{\ell}{x}p^x(1-p)^{\ell-x} \\
        & = (2k)^x\binom{\ell}{x}\log(\ell)^x \ell^{-x}\Big(1-\frac{2k\log(\ell)}{\ell} \Big)^{\ell-x} \\
        & \leq (2k)^x\frac{\ell^x}{x!}\log(\ell)^x \ell^{-x} \exp\Big(-2k\log(\ell) + \frac{2kx\log(\ell)}{\ell}\Big) \\
        & = \frac{(2k\log(\ell))^x}{x!}\ell^{(\frac{x}{\ell}-1)2k}.
    \end{align*}

    Let us now consider the probability of $E_{v,j}$ conditioned on the value of $X$.
    There are $\binom{(k-1)\ell-1}{k-1}$ choices of vertices $u_1,\dots, u_{k-1}$,
    and given a fixed choice of vertices $u_1,\dots, u_{k-1}$ 
    the probability of $N(v) \cap S_j \subseteq \Big(\bigcup_{t=1}^{k-1} N(u_t) \Big)$
    is at most $((k-1)p)^{X} = (\frac{2(k-1)k\log(\ell)}{\ell})^X$.

    By the law of total probability and the union bound 
    \[
     \mathbb{P}( E_{v,j}) \leq \binom{(k-1)\ell-1}{k-1} \sum_{x=0}^\ell \mathbb{P}\Big( E_{v,j} | X = x\Big)\mathbb{P}\Big(X = x\Big).
    \]
    So the product $\mathbb{P}( E_{v,j} | X = x)\mathbb{P}(X = x)$ is of interest.
    Let us consider this quantity; by the inequalities we have already shown
    \begin{align*}
        \mathbb{P}\Big( E_{v,j} | X = x\Big)\mathbb{P}\Big(X = x\Big) & \leq \Bigg(\frac{2(k-1)k\log(\ell)}{\ell} \Bigg)^x \Bigg(\frac{(2k\log(\ell))^x}{x!}\ell^{(\frac{x}{\ell}-1)2k} \Bigg) \\
        & < \ell^{-2k}\Bigg(\frac{2^{2x} k^{3x} \log(\ell)^{2x} \ell^{(\frac{1}{\ell}-1)x}}{x!} \Bigg).
    \end{align*}
    Plugging this into the earlier sum, and recalling the series definition of $e^z = \sum_{x=0}^\infty \frac{z^x}{x!}$,
    \begin{align*}
        \mathbb{P}( E_{v,j}) & \leq \binom{(k-1)\ell-1}{k-1} \sum_{x=0}^\ell \ell^{-2k}\Bigg(\frac{2^{2x} k^{3x} \log(\ell)^x \ell^{(\frac{1}{\ell}-1)x}}{x!} \Bigg) \\
        & < \ell^{-2k} \binom{(k-1)\ell-1}{k-1}  \sum_{x=0}^\infty \frac{\Big(4 k^{3} \log(\ell)^2 \ell^{(\frac{1}{\ell}-1)}\Big)^x}{x!} \\
        & = \ell^{-2k} \binom{(k-1)\ell-1}{k-1} \exp\Big(4 k^{3} \log(\ell)^2 \ell^{(\frac{1}{\ell}-1)}\Big) \\
        & < \ell^{-2k} \frac{(k\ell)^{k-1}}{(k-1)!} \exp\Big(4 k^{3} \log(\ell)^2 \ell^{(\frac{1}{\ell}-1)}\Big) \\
        & = \frac{k^{k-1}}{(k-1)!}\exp\Big(4 k^{3} \log(\ell)^2 \ell^{(\frac{1}{\ell}-1)}\Big)\ell^{-k-1}.
    \end{align*}
    Since $k$ is constant,
    observe the exponential term tends to $1$ as $\ell$ tends to infinity.
    Recalling $\mathbb{P}(E_v) \leq (k-1)\mathbb{P}(E_{v,j})$,
    we conclude that 
    \[
    \mathbb{P}(E_v) < \frac{k^{k}}{(k-1)!}\exp\Big(4 k^{3} \log(\ell)^2 \ell^{(\frac{1}{\ell}-1)}\Big)\ell^{-k-1}.
    \]
    Applying the union bound over all choice of $v$,
    the probability that there exists a vertex $v$ where event $E_v$ occurs is at most 
    \[
    \frac{k^{k+1}}{(k-1)!}\exp\Big(4 k^{3} \log(\ell)^2 \ell^{(\frac{1}{\ell}-1)}\Big)\ell^{-k}.
    \]
    This concludes the proof.
\end{proof}

We are now prepared to prove Theorem~\ref{Thm: Cop = Clique = k for all k}.
The bound  in Theorem~\ref{Thm: Cop = Clique = k for all k} may be far from tight, and no effort is made to optimizes the coefficient of $11$.
Notice that implicit in the following proof is that there exists infinity many graphs with $c(G) = \theta(G) = k$ for every $k$.

\begin{proof}[Proof of Theorem~\ref{Thm: Cop = Clique = k for all k}]
    When $k = 1$, $K_1$ is a graph with $c(K_1) = \theta(K_1) = 1$ on $1$ vertex, and when $k=2$, $C_4$ is a graph with $c(C_4) = \theta(C_4) = 2$ on $4$ vertices.
    Suppose then that $k\geq 3$ is constant.
    Let $\ell\geq 2$ be an integer which we will choose later, and let $p = \frac{2k\log(\ell)}{\ell}$.
    Let $H = H(k,\ell,p)$.
    By Lemma~\ref{Lemma: Max-degree random}, Lemma~\ref{Lemma: Clique = k in H}, and Lemma~\ref{Lemma: Always an escape in H} if $\ell$ is chosen so that 
    \begin{enumerate}
        \item $5k(k-1)\log(\ell) \leq \frac{\ell}{k-1}$, and
        \item $1 - k\ell^{1-\frac{3(k-1)k}{2}} > \frac{1}{2}$, and
        \item $1 - \frac{k^{k+1}}{(k-1)!}\exp\Big(4 k^{3} \log(\ell)^2 \ell^{(\frac{1}{\ell}-1)}\Big)\ell^{-k} > \frac{1}{2}$,
    \end{enumerate}
    then there is positive probability that $H$ simultaneously satisfies $\Delta(H) < \frac{k\ell}{k-1} - 1$, and $\theta(H) = k$, and no event $E_v$, as described in Lemma~\ref{Lemma: Always an escape in H}, occurs.

    One can verify, either analytically or using software, that if $\ell = 11\log(k)^2k^3$, then inequalities (1)-(3) are satisfied for all $k\geq 3$.
    Suppose then that $\ell = 11\log(k)^2k^3$ and let $G = H$ such that $\Delta(H) < \frac{k\ell}{k-1} - 1$, and $\theta(H) = k$, and no event $E_v$, as described in Lemma~\ref{Lemma: Always an escape in H}, occurs.
    Such a graph $G$ exists, since there is positive probability that $H$ has all of these properties simultaneously.
    Notice that $|V(G)| = k\ell = 11\log(k)^2k^4$.
    
    Since $\theta(G) = k$, $c(G) \leq k$.
    We claim that $c(G) > k-1$.
    Consider the Cops and Robbers game played with $k-1$ cops on $G$.
    Call these cops $C^{1}, \dots, C^{k-1}$ and 
    let $w_1, \dots, w_{k-1}$ be the initial positions of cops $C^{1}, \dots, C^{k-1}$.
    We will demonstrate a strategy for the robber that indefinitely evades capture.

    We begin by noting that $w_1, \dots, w_{k-1}$ is not a dominating set.
    To see this recall that $\Delta(G) < \frac{k\ell}{k-1} - 1$ and $|V(G)| = k\ell$.
    Hence, 
    \[
    \Big|\bigcup_{i=1}^{k-1} N[w_i]\Big| \leq (k-1)\Delta(G) < k\ell = |V(G)|.
    \]
    Thus, there exists a vertex $z \notin \bigcup_{i=1}^{k-1} N[w_i]$.
    Suppose the robber begins on $z$.
    Then the cops cannot capture the robber on their first turn, and the robber is guaranteed the opportunity to move at least once.

    Now, without loss of generality consider the game when it is the robber's opportunity to move.
    If there is no cop adjacent to the robber, then the robber does not move, and  the robber is guaranteed the opportunity to move at least once more.
    Suppose then that there is a cop adjacent to the robber.

    Let $u_1,\dots, u_{k-1}$ be the current location of the cops.
    Notice that multiple cops may be at the same vertex, but that this has no event on the following argument.
    Since there are strictly less than  $k$ cops, there is a $j$ such that $\{u_1,\dots, u_{k-1}\} \cap S_j = \emptyset$.
    Let the current location of the robber be $v \in S_i$.
    We consider the cases $j = i$ and $j\neq i$ separately.

    First consider if $j = i$.
    Then, no cop is on the clique $S_i$ which contains the robber.
    Then every cop has $\ell-1$ neighbours in a clique $S_q$ distinct from $S_i$.
    Since $\Delta(G) < \frac{k\ell}{k-1} - 1$, we conclude that
    \[
    \Big|\bigcup_{t=1}^{k-1} N[u_t] \cap S_i\Big| <  (k-1)\Big(\frac{k\ell}{k-1} - 1 - (\ell - 1) \Big) = \ell.
    \]
    Hence there exists a vertex $v' \in S_i $ not adjacent to any cop.
    Since the robber's current vertex $v$ is adjacent to a cop $v\neq v'$.
    So the robber moves to $v'$, which is possible since $S_i$ is a clique, and the robber is guaranteed the opportunity to move at least once more.

    Next, consider if $j\neq i$.
    Since event $E_v$ does not occur, and since no cop is in $S_j$ where $j\neq i$,
    \[
    N(v) \cap S_j \not\subseteq \Big(\bigcup_{t=1}^{k-1} N[u_t] \Big).
    \]
    Hence there exists a vertex $v' \in N(v) \cap S_j $ not adjacent to any cop.
    So the robber moves to $v'$, which is possible since $v$ and $v'$ are adjacent, and the robber is guaranteed the opportunity to move at least once more.

    We conclude that in every possible situation the robber has a strategy to ensure one more move before being captured.
    Hence, they can evade capture indefinitely.
    Thus, $k-1$ cops do not win, implying $c(G) = k = \theta(G)$.
    This completes the proof.
\end{proof}

\section{Finding Induced Cycles When $c(G) = \theta(G)$}

The goal of this section is to prove Theorem~\ref{Thm: Clique Cover = Cop Structure}.
Our proof proceeds by demonstrating that if $G$ is a graph with $c(G) = \theta(G) = k\geq 3$, 
then for all $3 \leq t \leq k+1$, $k-1$ cops can play so that the 
robber's only winning strategy is to take a walk which induces a cycles of length $t$.
Given  $c(G) > k - 1$, such a strategy must be possible for the robber to execute implying the existence of an induced $C_t$. 

To show the cops' strategy is feasible,
we must prove certain structure appears in graphs that satisfy $c(G) = \theta(G)\geq 3$.
For vertex subsets $S$ and $T$, we let $E(S,T)$ be the set of all edges with one end in $S$ and the other in $T$.

\begin{lemma}\label{Lemma: Complete-Clique-Decomposition}
    If $G$ is a connected graph with $c(G) = \theta(G) = k\geq 3$, and $K^{(1)}, \dots, K^{(k)}$ is a clique cover of $G$,
    then for all distinct $K^{(i)}, K^{(j)}$, the edge set $E(K^{(i)},K^{(j)})\neq \emptyset$.
\end{lemma}

\begin{proof}
    Suppose $G$ admits a clique cover $K^{(1)}, \dots, K^{(k)}$ such that for some $K^{(i)}, K^{(j)}$, the edge set $E(K^{(i)},K^{(j)})= \emptyset$.
    Without loss of generality suppose $i=k-1$ and $j = k$.
    We claim that $c(G) < k$.

    To show this, we will provide a winning strategy for $k-1$ cops.
    For each $1\leq i \leq k-1$, begin a cop, call it $C^i$ in $K^{(i)}$.
    Since $K^{(1)}, \dots, K^{(k)}$ is a clique cover, the set of vertices non-adjacent to cops at the start of the game is a subset of $K^{(k)}$.
    So if the robber is not captured on their first turn they must being in $K^{(k)}$.
    Suppose then the robber beings on $K^{(k)}$.

    Given $E(K^{(k-1)},K^{(k)})= \emptyset$ if the robber wants to walk from it's current vertex in $K^{(k)}$ to $K^{(k-1)}$, then it must first enter a different clique $K^{(i)}$.
    Hence, if cops $C^1,\dots, C^{k-2}$ do not move from their respective cliques, the robber will be unable to leave $K^{(k)}$.
    This leaves the cop $C^{k-1}$ free to walk from $K^{(k-1)}$ to $K^{(k)}$ which is possible since $G$ is connected.
    Once $C^{k-1}$ arrives in $K^{(k)}$, they can capture the robber, since $K^{(k)}$ is a clique, and the robber cannot leave $K^{(k)}$ without being captured by one of the other cops.
\end{proof}

\begin{lemma}\label{Lemma: Robber First Move}
    Let $G$ be a connected graph with $c(G) = \theta(G) = k\geq 3$, and let $K^{(1)}, \dots, K^{(k)}$ be a clique cover of $G$.
    If $k-1$ cops, $C^i$ for $1\leq i \leq k-1$, occupy vertices $v_i \in K^{(i)}$ such that $N(v_i) \cap K^{(k)} \neq \emptyset$ at the start of the cops' turn,
    then a robber with a winning strategy must occupy a vertex $u \in K^{(k)}$ such that for all $i$, $N(u) \cap K^{(i)} \neq \emptyset$.
\end{lemma}

\begin{proof}
    Let $G$ be a connected graph with $c(G) = \theta(G) = k$, and let $K^{(1)}, \dots, K^{(k)}$ be a clique cover of $G$.
    Suppose at the start of the current cop turn, that for all $1\leq i \leq k-1$,
    a cop $C^i$ occupies a vertex $v_i \in K^{(i)}$ such that $N(v_i) \cap K^{(k)} \neq \emptyset$.
    Also suppose that the robber occupies a vertex $w$ such that $w \notin K^{(k)}$
    or $w \in K^{(k)}$ but there exists a $j$ such that, $N(u) \cap K^{(j)} =  \emptyset$.

    If $w \notin K^{(k)}$, then $w \in K^{(j)}$ for some $j \neq k$, since  $K^{(1)}, \dots, K^{(k)}$ is a clique cover of $G$.
    So the cop $C^j$ can capture the robber, given it is the cops turn and $C^j$ is also in the clique $K^{(j)}$.
    So such a case does not occur if the robber has a winning strategy.

    Otherwise, $w \in K^{(k)}$ but there exists a $j$ such that, $N(u) \cap K^{(j)} = \emptyset$.
    In this case, the cop $C^j$ moves to $K^{(k)}$ on their turn, while all other cops do not move.
    This is possible by our assumption $C^j$ begins on vertex $v_j$ which has neighbours in $K^{(k)}$.
    In response to this, the robber cannot move to a clique $K^{(i)}$ where $i\neq j,k$ without being captured, since cop $C^i$ will capture on the next cop turn.
    Similarly the robber cannot stay on $K^{(k)}$ without being captured by $C^j$ on the next cop turn.
    But moving directly to $K^{(j)}$ is impossible since we supposed $N(u) \cap K^{(j)} = \emptyset$.
    So again, the robber is captured, implying this case does not occur if the robber has a winning strategy.
    This completes the proof.
\end{proof}

\begin{lemma}\label{Lemma: There is a triangle}
    If $G$ is a connected graph with $c(G) = \theta(G) = k\geq 3$, then $G$ contains a triangle.
\end{lemma}

\begin{proof}
    Suppose $G$ is a connected triangle-free graph with $c(G) = \theta(G) = k\geq 3$.
    Then $G$ has a clique cover $K^{(1)}, \dots, K^{(k)}$ where every clique $K^{(i)}$ is isomorphic to $K_1$ or a $K_2$.
    Let  $K^{(1)}, \dots, K^{(k)}$ be such a clique cover.

    We claim prove $k-1$ cops can capture the robber in $G$.
    If each $K^{(i)}$ is isomorphic to $K_1$, then this is trivial, since the cops can stand on all but one vertex in the graph.
    Suppose then, without loss of generality that $K^{(k)}$ is isomorphic to $K_2$.
    By Lemma~\ref{Lemma: Complete-Clique-Decomposition}, for all $i\neq j$, $E(K^{(i)}, K^{(j)}) \neq \emptyset$.
    Then, there exists a vertex $v_1 \in K^{(1)}$ such that $N(v_1)\cap K^{(k)} \neq \emptyset$, let $v_1$ be such a vertex.
    Thus, there exists a vertex $w \in K^{(k)}$ adjacent to $v_1$.
    Let $w$ be such a vertex and let $u$ be the vertex in $K^{(k)}$ that is not $w$.
    Should $K^{(1)}$ be isomorphic to $K_2$ let $u'$ be the vertex in $K^{(1)}$ that is not $v_1$.
    For all $i\neq 1,k$, let $v_i \in K^{(i)}$ be a vertex with at least one neighbour in $K^{(1)}$.
    Such vertices exist by Lemma~\ref{Lemma: Complete-Clique-Decomposition}.
    
    Let $C^{1}, \dots, C^{k-1}$ be our set of cops
    and suppose that each cop $C^{(i)}$ begins on vertex $v_i$.
    The only vertex in $G$ which is possibly non-adjacent to a cop is $u$, so we suppose the robber starts at $u$.
    If $u$ is adjacent to $v_1$, then the starting position of the cops forms a dominating set, so suppose $u$ and $v_1$ are non-adjacent.
    On their first move the cops proceed via the following strategy:
    cop $C^{1}$ moves from $v_1$ to $w$, 
    while each cop $C^{i}$ where $i\neq 1$, moves from $v_i$ to $v_i'$ (which might be the same vertex),
    where $v'_i \in K^{(i)}$ 
    is a vertex defined as follows
    \begin{enumerate}
        \item if $u'$ does not exist, then $v'_i = v_i$, or
        \item if $u'$ exists and $N(u') \cap K^{(i)} \neq \emptyset$, then let $v'_i \in K^{(i)}$ be a neighbour of $u'$, or
        \item if $u'$ exists and $N(u) \cap K^{(i)} = \emptyset$, then let $v'_i = v_i$.
    \end{enumerate}
    
    Now we consider how the robber might respond. If the robber does not move or moves to $w$, then on the next cop turn the cop $C^{1}$ can capture the robber.
    If the robber moves to a clique $K^{(i)}$ where $i\neq 1,k$, then the cop $C^{i}$ can capture the robber on the next cop turn since $K^{(i)}$ is a clique.
    Thus, the robber must move to $K^{(1)}$.
    Recall $u$ and $v_1$ are non-adjacent,
    hence the robber must move to $u'$.
    So if $u'$ does not exist, then the cops have captured the robber.
    Otherwise, $u'$ exists implying every cop $C^{i}$ where $i\neq 1$ moved using case (2) or (3).
    If any cop $C^{i}$ where $i\neq 1$ moved using case (2), then this cop can capture the robber on the next cop turn.
    So suppose that all cops $C^{i}$ where $i\neq 1$ moved using case (3).

    In this case every clique $K^{(i)}$ where $i\neq 1$ contains a cop that occupies a vertices adjacent to some non-empty part of $K^{(1)}$, 
    while the robber occupies a vertex in $K^{(1)}$
    that has no neighbours in $K^{(2)}$, since $1 < 2< k$.
    Additionally, it is the cops turn to move.
    Since the labels of the cliques in the clique cover are arbitrary, Lemma~\ref{Lemma: Robber First Move} implies the cops have a strategy to capture the robber.
    So, $c(G) \leq k-1$ contradicting our assumption that $c(G) = k$.
    Therefore, no such graph $G$ exists, concluding the proof.
\end{proof}

\begin{lemma}\label{Lemma: There is a C4}
    If $G$ is a connected graph with $c(G) = \theta(G) = k\geq 3$, then $G$ contains an induced $C_4$.
\end{lemma}

\begin{proof}
    Let $G$ be a connected graph with $c(G) = \theta(G) = k\geq 3$. Then $G$ has a clique cover $K^{(1)}, \dots, K^{(k)}$.
    Let  $K^{(1)}, \dots, K^{(k)}$ be such a clique cover.

    We claim that $c(G) > k-1$ implies $G$ contains an induced $C_4$.
    To see this we start $k-1$ cops $C^{i}$ on vertices
    $v_i \in K^{(i)}$ for all $1\leq i \leq k-1$.
    We suppose without loss of generality, see Lemma~\ref{Lemma: Complete-Clique-Decomposition}, that $N(v_1) \cap K^{(k)} \neq \emptyset$.
    Let $w \in N(v_1) \cap K^{(k)}$.
    Since $c(G) > k-1$ the robber has a winning strategy, and trivially, the cops' starting position forces the robber to being on a vertex $u\in K^{(k)}\setminus N(v_1)$.

    Consider the following first move by the cops.
    All cops other than $C^{1}$ remain on their current vertex, and cop $C^{1}$ moves to $w$.
    To avoid capture, the robber must move to a vertex $u'$ outside of $N(w) \cup K^{(2)} \cup \dots K^{(k-1)}$.
    Since $c(G) > k-1$ such a vertex $u'$ exists, and trivially, $u' \in K^{(1)}$.
    Since $u' \notin N(w)$ and $u \notin N(v_1)$, the vertices $\{v_1,w,u,u'\}$ induce $C_4$.
    This completes the proof.
\end{proof}

We are now prepared to prove Theorem~\ref{Thm: Clique Cover = Cop Structure}.

\begin{proof}[Proof of Theorem~\ref{Thm: Clique Cover = Cop Structure}]
Let $k\geq 3$ be a fixed but arbitrary integer, and let $G$ be a connected graph with $c(G) = \theta(G) = k$.
We aim to show that for all $3 \leq t\leq k+1$, $G$ contains an induced $C_t$.
If $k = 3$, then Lemma~\ref{Lemma: There is a triangle} and Lemma~\ref{Lemma: There is a C4} imply $G$ contains an induced $C_3$ and $C_4$ respectively, so the claim is true.
Suppose then that $k\geq 4$ while the claim holds for graphs with cop number equal to clique cover number equal to $k-1$.

Let $K^{(1)},\dots, K^{(k)}$ be a clique cover of $G$,
and let $G' = G - K^{(k)}$.
One can easily verify $\theta(G') = \theta(G) - 1 = k-1$.
Hence, $c(G') \leq k-1$.
To see that $c(G')\geq k-1$, notice that if $k-2$ cops can catch the robber on $G'$, then $k-1$ cops can catch the robber on $G$, 
by reserving one cop to stand on $K^{(k)}$, thereby preventing the robber from entering $K^{(k)}$, while the remaining $k-2$ cops capture the robber on the vertices of $G'$.
Thus, $c(G') = k-1$.

By induction, for every $3 \leq t \leq k$, $G'$ contains an induced $C_t$.
Given $G'$ is an induced subgraph of $G$, this implies $G$ contain an induced $C_t$ for every $3 \leq t \leq k$.
Then all that remains for us to show is that $G$ contains an induced $C_{k+1}$.

    We claim that $c(G) > k-1$ implies $G$ contains an induced $C_{k+1}$.
    Suppose for contradiction that there exists an $i$ where for all vertices $v \in K^{(i)}$, there exists a $j$ such that $N(v)\cap K^{(j)} = \emptyset$.
    Without loss of generality with respect to the labels of the cliques in our clique cover, suppose $K^{(1)}$ is a clique with this property.
    Then, Lemma~\ref{Lemma: Robber First Move} implies 
    that 
    if $k-1$ cops, $C^i$ for $1\leq i \leq k-1$, occupy vertices $v_i \in K^{(i+1)}$ such that $N(v_i) \cap K^{(1)} \neq \emptyset$ at the start of the game,
    then the robber will lose, since there is no vertex 
    $u \in K^{(1)}$ such that for all $i$, $N(u) \cap K^{(i)} \neq \emptyset$.
    Observe that by Lemma~\ref{Lemma: Complete-Clique-Decomposition} it is possible for $k-1$ cops, $C^i$ for $1\leq i \leq k-1$ to occupy vertices $v_i \in K^{(i+1)}$ such that $N(v_i) \cap K^{(1)} \neq \emptyset$ at the start of the game.
    Thus, we have shown $c(G) \leq k-1$, contradicting $c(G) = k$.
    So, for all $i$ there exists a vertex $v_i \in K^{(i)}$, where for all $j$, $N(v_i)\cap K^{(j)} = \emptyset$.
    For $1\leq i \leq k-1$ let $v_i \in K^{(i+1)}$ be such a vertex.

    Imagine we start a cop $C^{i}$ on vertex
    $v_i \in K^{(i+1)}$ for all $1\leq i \leq k-1$.
    Since $c(G) > k-1$, the robber has a winning strategy when playing against these cops.
    Suppose the robber begins on a vertex $u_1$ while following a winning strategy.
    Then by our choice of $v_1,\dots, v_{k-1}$, Lemma~\ref{Lemma: Robber First Move} implies $u_1 \in K^{(1)}$ and for all $j$, $N(u_1) \cap K^{(j)} \neq \emptyset$.

    Let $1\leq t \leq k$.
    We imagine the game proceeds in rounds, where a round beings on the cop's turn, a series of moves by both players occurs, then the round concludes following a robber turn.
    The vertex $u_{t}$ denotes the location of the robber at the start of round $t$.
    Suppose at the start of round $1$, each cop $C^i$ is located at $v_i$. 
    This describes the starting position of the game.
    We will demonstrate a cop strategy that ensures at the start of round $t>1$,  that for all $1 \leq i \leq t-2$, cop $C^i$ is located at $u_i$, and for all $j\geq t$, cop $C^j$ is located at $v_j$,
    while we suppose that cop $C^{t-1}$ occupies a vertex $x_{t-1}\in K^{(t-1)}$ adjacent to $v_{t-1}$.
    We induct, on the round number, so that for all $1\leq i \leq t$, and $u_{i} \in K^{(i)}$ and for all $j$, $N(u_{i}) \cap K^{(j)} \neq \emptyset$.
    We have already established this claim holds when $t=1$.

    With the game thus far having been played this way,
    we describe the cops' strategy in round $1\leq t < k$.
    The cops move as follows. 
    Cop $C^{t-1}$ moves from $x_{t-1}$ to $u_{t-1}$ which is possible since both vertices are in $K^{(t-1)}$ and $K^{(t-1)}$ is a clique.
    Cop $C^t$ moves from $v_t$ to some neighbour of $v_t$ in $K^{(t)}$ which we call $x_t$. 
    Such a vertex exists by our choice of $v_t$.
    Since the cops cannot capture the robber, $x_t \neq u_t$.
    All other cops do not move.
    In response Lemma~\ref{Lemma: Robber First Move} implies the robber must move to a vertex $u \in K^{(t+1)}$ where for all $j$, $N(u) \cap K^{(j)} \neq \emptyset$.
    Following the robber's move to $u$, the cops end the round which sets $u_{t+1} = u$.
    This satisfies our induction statement.

    Recalling that $c(G) > k-1$, the robber will not be captured during any round, regardless of the cops' strategy.
    Since the robber will not be captured on the cops' turn during any round,
    $u_t$ is non-adjacent to all vertices $u_1,\dots, u_{t-2}$, given those vertices are occupied by cops at the start of round $t$.
    The cops' strategy also implies that each $u_t$ is adjacent to $u_{t+1}$, for $1\leq t < k$.
    Thus, $\{u_1,\dots, u_{k}\}$ induces $P_k$.

    Now consider the game in round $k$.
    Cops occupy vertices $u_1,\dots, u_{k-2},x_{k-1}$ while the robber occupies vertex $u_k \in K^{(k)}$.
    The cops move as follows. 
    Cop $C^{k-1}$, which currently occupies vertex $x_{k-1} \in K^{(k-1)}$, moves to $u_{k-1}$
    which is possible since both vertices are in $K^{(k-1)}$ and $K^{(k-1)}$ is a clique.
    Cop $C^1$, which currently occupies vertex $u_1$, moves to a vertex $x_k \in K^{(k)}$.
    Such a vertex exists by our proof that for all $j$, $N(u_1)\cap K^{(j)} \neq \emptyset$.
    All other cops do not move.
    In response Lemma~\ref{Lemma: Robber First Move} implies the robber must move to a vertex $w \in K^{(1)}$.
    Since the robber is not captured on the next cop turn, $w$ is not adjacent to any vertex $u_2,\dots, u_{k-1}$.
    So $w\neq u_1$, since $u_1$ and $u_2$ are adjacent.
    Now $w$ is adjacent to $u_1$, since they lie in a clique together, and $u_k$ is adjacent to $w$ since the robber moves from $u_k$ to $w$.
    Since, $\{u_1,\dots, u_{k}\}$ induces $P_k$ such that each $u_t$ is adjacent to $u_{t+1}$,
    we conclude that $\{u_1,\dots, u_{k},w\}$ induces $C_{k+1}$.
    This completes the proof.
\end{proof}

\section{Typical Graphs with No Short Induced Cycle}

In this section we will prove Theorem~\ref{Coro: typical C_l-free Graphs}.
That is, under the assumption that almost all $3K_1$-free graphs and $2K_2$-free graphs have $\chi = \omega$,
then for $\ell$ equal to $3$ or $4$, we will prove that almost every $C_\ell$-free graph $G$ has $c(G) < \alpha(G)$.
This assumption is not know to be true in the literature, 
however it is reasonable to expect it may be true given  the following theorem of Reed and Yuditsky \cite{reed2025asymptotic}.

\begin{theorem}\label{Thm: R&Y}
    Let $H$ be a tree or a cycle $C_\ell$ where $\ell\neq 6$, then almost every $H$-free graph $G$ has $\chi(G) = \omega(G)$.
\end{theorem}

To prove Theorem~\ref{Coro: typical C_l-free Graphs} we require some preliminary lemmas.
We denote the Ramsey number for parameters $k$ and $t$ by $R(k,t)$.

\begin{lemma}\label{Lemma: alpha>2 Tri-free}
    All triangle-free graphs with at least $6$ vertices have $\alpha(G) \geq 3$.
\end{lemma}

\begin{proof}
    If $G$ is a triangle-free graph with independence number at most $2$, then $G$ has $\omega(G),\alpha(G) \leq 2$.
    Thus, $|V(G)|< R(3,3)$.
    It is well known, and elementary to show, that $R(3,3)= 6$.
    Hence, such a graph $G$ has at most $5$ vertices.
\end{proof}

\begin{lemma}\label{Lemma: theta=2 C4-free}
    If $G$ is a connected graph $C_4$-free graph with $\theta(G)\leq 2$, then $c(G) = 1$.
\end{lemma}

\begin{proof}
    Suppose the statement is false.
    Let $G$ be a smallest connected $C_4$-free graph with $\theta(G) \leq 2$ and $c(G) \neq 1$.
    Then $\theta(G) =2$ since $1< c(G) \leq \theta(G) \leq 2$.
    Furthermore, for all induced subgraphs $H$ of $G$ not equal to $G$,  $c(H) = 1$, since $H$ is also $C_4$-free.

    A standard argument in Cops and Robbers, see for instance \cite{bonato2011game}, 
    is that if there exists vertices $u$ and $v$ where $N[u]\subseteq N[v]$ in a graph $Q$,
    then $c(Q) = c(Q-u)$.
    In this case the vertex $u$ is called a corner, and it is never advantageous for either player to move to $u$, hence removing $u$ does not effect the game.
    Since $G$ is a smallest counter-example $G$ contains no corner.

   Let $K^{(1)},K^{(2)}$ be a clique cover of $G$. If $K^{(1)}$ or $K^{(2)}$ is isomorphic to $K_1$, then $G$ trivially contains a corner.
   Hence, $|K^{(1)}|\geq 2$ and $|K^{(2)}|\geq 2$.
   Similarly, if there exists a vertex in either clique, without neighbours in the opposite clique, then this vertex is a corner.
   So every vertex in $K^{(1)}$ has a neighbour in $K^{(2)}$ and visa-versa.
   
   Let $u,v \in K^{(1)}$ be distinct vertices.
   Since $u$ and $v$ are both in $K^{(1)}$, while $G$ contains no corner, $N(u) \not\subseteq N(v)$ and $N(v) \not\subseteq N(u)$.
   Hence there is a vertex $x \in N(u) \setminus N(v)$ and a vertex $y \in N(v) \setminus N(u)$.
   Immediately, $x,y \in K^{(2)}$, so $x$ and $y$ are adjacent.
   But this implies $\{u,x,y,v\}$ induces $C_4$, contradicting that $G$ is $C_4$-free.
   We conclude that every $C_4$-free graph with $\theta(G) \leq 2$ has $c(G) = 1$.
\end{proof}

We are now prepared to prove Theorem~\ref{Coro: typical C_l-free Graphs}.
\begin{proof}[Proof of Theorem~\ref{Coro: typical C_l-free Graphs}]
Assume that almost all $3K_1$-free graphs and $2K_2$-free graphs have $\chi = \omega$.
We consider triangle-free graphs and $C_4$-free graphs separately.

Suppose $G$ is a connected triangle-free graph with $\theta(G) = \alpha(G) \geq 3$.
Since $G$ has no triangle, Lemma~\ref{Lemma: There is a triangle} implies that $c(G) < \theta(G) = \alpha(G)$.
Notice that $\bar{C_3} = 3K_1$.
Since Lemma~\ref{Lemma: alpha>2 Tri-free} implies all but finitely many triangle-free graphs $G$ have $\alpha(G)\geq 3$,
and since we assume almost all $3K_1$-free graphs have $\chi = \omega$, 
almost every triangle-free graph $G$ has $\theta(G) = \chi(\bar{G}) = \omega(\bar{G}) = \alpha(G)$.
Thus,
we have shown almost every triangle-free graphs has $c(G) < \alpha(G)$.

Now consider a connected $C_4$-free graph with $\theta(G) = \alpha(G) \geq 2$.
If $\alpha(G)\geq 3$, then since $G$ has no induced $C_4$, Lemma~\ref{Lemma: There is a C4} implies that $c(G) < \theta(G) = \alpha(G)$.
Otherwise $\alpha(G) = \theta(G) = 2$, and Lemma~\ref{Lemma: theta=2 C4-free} implies $c(G) = 1 < 2 = \alpha(G)$.
Trivially, the set of graphs with independence number $1$, that is cliques, make up a negligible (i.e. tending to zero as $n$ tends to infinity) proportion of $n$ vertex $C_4$-free graphs.
Notice that $\bar{C_4} = 2K_2$.
Then, since we assume almost all $2K_2$-free graphs have $\chi = \omega$, 
almost every $C_4$-free graph $G$ has $\theta(G) = \chi(\bar{G}) = \omega(\bar{G}) = \alpha(G)\geq 2$.
Thus, almost all $C_4$-free graphs $G$ have $c(G) < \alpha(G)$.
This completes the proof.
\end{proof}

\section{Future Work}

We conclude the paper by leaving some problems for future work.
The first problem concerns perfect graphs.
Given Theorem~\ref{Coro: Perfect Graphs}
we know there is no perfect graph with cop number equal to independence number greater than or equal to $4$.
We also know that 
$C_4$ is a perfect graph with cop number and clique cover number $2$.

\begin{problem}
    Prove that every connected perfect graph $G$ with $\alpha(G) = 3$
    is $2$-cop-win, or demonstrate a perfect graph $G$ with $c(G) = \alpha(G) = 3$.
\end{problem}

The next problem also focuses on perfect graphs.
Here we are interested in how much the upper bound $c(G) \leq \alpha(G)-1$ can be improved for perfect graphs with large independence number.
To motivate this problem observe that bipartite graphs (which are necessarily perfect) all have cop number at most $o(\alpha)$.
To see this note that $\alpha\geq n/2$ for all bipartite graphs, while it has been shown by numerous authors that all connected graphs have cop number at most $o(n)$.

\begin{problem}
    Given an integer $k$, what is the least integer $f(k)$ such that every connected perfect graph $G$ with $\alpha(G) = k$,
    has $c(G) \leq f(k)$?
\end{problem}

We are also interested in determining if almost every graph with a forbidden induced cycle has cop number strictly less than independence number.
To this end we make the following conjecture.

\begin{conjecture}
    For all $\ell\geq 3$, almost every $C_\ell$-free graph $G$ has $c(G) < \alpha(G)$.
\end{conjecture}

The final problem we propose is to determine, as a function of $k$, the order of a smallest graph $G$ with $c(G) = \theta(G) = k$.
Finding an exact value when $k$ is large seems impossible, so we propose the following easier problem.

\begin{problem}
    Let $f(k)$ be the least integer such that there exists a graph $G$ with order $f(k)$ and $c(G) = \theta(G) = k$.
    What is the least constant $c$ such that $f(k) = O(k^{c})$?
\end{problem}

\section*{Acknowledgement}
The authors would like to thank Bruce Reed and Yelena Yuditsky for their helpful feedback concerning 
if almost every $F$-free graph $G$, where $F$ is a forest, satisfies $\chi(G) = \omega(G)$.
Clow is supported by the Natural Sciences and Engineering Research Council of Canada (NSERC) through PGS D-601066-2025
while
Zaguia is supported by the Royal Military College of Canada and NSERC Canada DDG-2023-00038.

\bibliographystyle{abbrv}
\bibliography{bib}

\end{document}